\PassOptionsToPackage{unicode}{hyperref}
\PassOptionsToPackage{hyphens}{url}
\PassOptionsToPackage{dvipsnames,svgnames,x11names}{xcolor}
\documentclass[
]{article}
\usepackage{xcolor}
\usepackage{amsmath,amssymb}
\usepackage{iftex}
\ifPDFTeX
  \usepackage[T1]{fontenc}
  \usepackage[utf8]{inputenc}
  \usepackage{textcomp} 
\else 
  \usepackage{unicode-math} 
  \defaultfontfeatures{Scale=MatchLowercase}
  \defaultfontfeatures[\rmfamily]{Ligatures=TeX,Scale=1}
\fi
\usepackage{lmodern}
\ifPDFTeX\else
\fi
\IfFileExists{upquote.sty}{\usepackage{upquote}}{}
\IfFileExists{microtype.sty}{
  \usepackage[]{microtype}
  \UseMicrotypeSet[protrusion]{basicmath} 
}{}
\makeatletter
\@ifundefined{KOMAClassName}{
  \IfFileExists{parskip.sty}{%
    \usepackage{parskip}
  }{
    \setlength{\parindent}{0pt}
    \setlength{\parskip}{6pt plus 2pt minus 1pt}}
}{
  \KOMAoptions{parskip=half}}
\makeatother
\makeatletter
\ifx\paragraph\undefined\else
  \let\oldparagraph\paragraph
  \renewcommand{\paragraph}{
    \@ifstar
      \xxxParagraphStar
      \xxxParagraphNoStar
  }
  \newcommand{\xxxParagraphStar}[1]{\oldparagraph*{#1}\mbox{}}
  \newcommand{\xxxParagraphNoStar}[1]{\oldparagraph{#1}\mbox{}}
\fi
\ifx\subparagraph\undefined\else
  \let\oldsubparagraph\subparagraph
  \renewcommand{\subparagraph}{
    \@ifstar
      \xxxSubParagraphStar
      \xxxSubParagraphNoStar
  }
  \newcommand{\xxxSubParagraphStar}[1]{\oldsubparagraph*{#1}\mbox{}}
  \newcommand{\xxxSubParagraphNoStar}[1]{\oldsubparagraph{#1}\mbox{}}
\fi
\makeatother

\usepackage{color}
\usepackage{fancyvrb}

\DefineVerbatimEnvironment{Highlighting}{Verbatim}{commandchars=\\\{\}}
\usepackage{framed}
\definecolor{shadecolor}{RGB}{241,243,245}
\newenvironment{Shaded}{\begin{snugshade}}{\end{snugshade}}

\newcommand{\AttributeTok}[1]{\textcolor[rgb]{0.40,0.45,0.13}{#1}}

\newcommand{\BuiltInTok}[1]{\textcolor[rgb]{0.00,0.23,0.31}{#1}}

\newcommand{\DecValTok}[1]{\textcolor[rgb]{0.68,0.00,0.00}{#1}}

\newcommand{\ErrorTok}[1]{\textcolor[rgb]{0.68,0.00,0.00}{#1}}
\newcommand{\ExtensionTok}[1]{\textcolor[rgb]{0.00,0.23,0.31}{#1}}

\newcommand{\KeywordTok}[1]{\textcolor[rgb]{0.00,0.23,0.31}{\textbf{#1}}}
\newcommand{\NormalTok}[1]{\textcolor[rgb]{0.00,0.23,0.31}{#1}}
\newcommand{\OperatorTok}[1]{\textcolor[rgb]{0.37,0.37,0.37}{#1}}

\usepackage{longtable,booktabs,array}
\usepackage{calc} 
\usepackage{etoolbox}
\makeatletter
\patchcmd\longtable{\par}{\if@noskipsec\mbox{}\fi\par}{}{}
\makeatother
\IfFileExists{footnotehyper.sty}{\usepackage{footnotehyper}}{\usepackage{footnote}}
\makesavenoteenv{longtable}
\usepackage{graphicx}
\makeatletter
\newsavebox\pandoc@box
\newcommand*\pandocbounded[1]{
  \sbox\pandoc@box{#1}%
  \Gscale@div\@tempa{\textheight}{\dimexpr\ht\pandoc@box+\dp\pandoc@box\relax}%
  \Gscale@div\@tempb{\linewidth}{\wd\pandoc@box}%
  \ifdim\@tempb\p@<\@tempa\p@\let\@tempa\@tempb\fi
  \ifdim\@tempa\p@<\p@\scalebox{\@tempa}{\usebox\pandoc@box}%
  \else\usebox{\pandoc@box}%
  \fi%
}
\def\fps@figure{htbp}
\makeatother

\NewDocumentCommand\citeproctext{}{}

\makeatletter
 \let\@cite@ofmt\@firstofone
 \def\@biblabel#1{}
 \def\@cite#1#2{{#1\if@tempswa , #2\fi}}
\makeatother
\newlength{\cslhangindent}
\newlength{\csllabelwidth}
\newenvironment{CSLReferences}[2] 
 {\begin{list}{}{%
  \setlength{\itemindent}{0pt}
  \setlength{\leftmargin}{0pt}
  \setlength{\parsep}{0pt}
  \ifodd #1
   \setlength{\leftmargin}{\cslhangindent}
   \setlength{\itemindent}{-1\cslhangindent}
  \fi
  \setlength{\itemsep}{#2\baselineskip}}}
 {\end{list}}
\usepackage{calc}

\providecommand{\tightlist}{%
  \setlength{\itemsep}{0pt}\setlength{\parskip}{0pt}}

\usepackage{arxiv}
\usepackage{orcidlink}
\usepackage{amsmath}
\usepackage[T1]{fontenc}
\makeatletter
\@ifpackageloaded{caption}{}{\usepackage{caption}}
\AtBeginDocument{%
\ifdefined\contentsname
  \renewcommand*\contentsname{Table of contents}
\else
  \newcommand\contentsname{Table of contents}
\fi
\ifdefined\listfigurename
  \renewcommand*\listfigurename{List of Figures}
\else
  \newcommand\listfigurename{List of Figures}
\fi
\ifdefined\listtablename
  \renewcommand*\listtablename{List of Tables}
\else
  \newcommand\listtablename{List of Tables}
\fi
\ifdefined\figurename
  \renewcommand*\figurename{Figure}
\else
  \newcommand\figurename{Figure}
\fi
\ifdefined\tablename
  \renewcommand*\tablename{Table}
\else
  \newcommand\tablename{Table}
\fi
}
\@ifpackageloaded{float}{}{\usepackage{float}}
\floatstyle{ruled}
\@ifundefined{c@chapter}{\newfloat{codelisting}{h}{lop}}{\newfloat{codelisting}{h}{lop}[chapter]}
\floatname{codelisting}{Listing}

\usepackage{amsthm}
\theoremstyle{plain}
\newtheorem{corollary}{Corollary}[section]
\theoremstyle{plain}
\newtheorem{proposition}{Proposition}[section]
\theoremstyle{plain}
\newtheorem{lemma}{Lemma}[section]
\theoremstyle{plain}
\newtheorem{theorem}{Theorem}[section]
\theoremstyle{remark}
\AtBeginDocument{}

\makeatother
\makeatletter
\@ifpackageloaded{caption}{}{\usepackage{caption}}
\@ifpackageloaded{subcaption}{}{\usepackage{subcaption}}
\makeatother
\usepackage{bookmark}
\IfFileExists{xurl.sty}{\usepackage{xurl}}{} 
\hypersetup{
  pdftitle={Random Triangles on Concentric Circles},
  pdfauthor={Brandon M. Greenwell},
  pdfkeywords={geometric probability, concentric circles, Pillow
Problem, random triangles, arcsine distribution},
  colorlinks=true,
  linkcolor={blue},
  filecolor={Maroon},
  citecolor={Blue},
  urlcolor={Blue},
  pdfcreator={LaTeX via pandoc}}

\newcommand{\runninghead}{A Preprint }
\renewcommand{\runninghead}{A preprint }
\title{Random Triangles on Concentric Circles}
\author{\textbf{Brandon M.
Greenwell}~\orcidlink{0000-0002-8120-0084}\\Department of Operations,
Business Analytics, and Information Systems\\University of Cincinnati
Carl H. Lindner College of Business\\Cincinnati,
Ohio,\ 45221\\\href{mailto:greenwbm@ucmail.uc.edu}{greenwbm@ucmail.uc.edu}}
\date{}
\begin{document}
\maketitle
\begin{abstract}
Lewis Carroll's Pillow Problem asks for the probability that three
points chosen at random in the plane form an obtuse triangle. The
question has no answer until the sampling scheme is pinned down, and the
cleanest way to pin it down is to put the points on a circle, which
gives 3/4. This paper solves the version in which the three vertices lie
on three concentric circles of radii \(r_1\), \(r_2\), and \(r_3\). The
answer is a sum of three terms, each the probability that a weighted sum
of two independent arcsine variables exceeds a threshold, and it
collapses to 3/4 when the radii are equal. Reading the formula gives
closed forms in two subfamilies, one an arcsine and one the Legendre chi
function; a Pythagorean condition deciding which vertex can carry the
obtuse angle; and the bound \(1/2 \leq P < 1\), with the minimum
attained only when one vertex sits at the common center and the other
two radii are equal. Conditioning on the radii shows that the same
formula is the kernel for every independent rotationally symmetric
sampling scheme, so the bound applies to all of them at once and
averaging over Rayleigh radii recovers the known Gaussian value of 3/4
in closed form, supplying a link between the circular and Gaussian
conventions that the literature records as missing. Simulation confirms
the results throughout.
\end{abstract}
{\bfseries \emph Keywords}
\def\sep{\textbullet\ }
geometric probability \sep concentric circles \sep Pillow
Problem \sep random triangles \sep 
arcsine distribution


\section{Introduction}\label{sec-intro}

Problem 58 of Carroll's \emph{Pillow Problems} (Carroll 1893) reads:
``Three points are taken at random on an infinite plane. Find the chance
of their being the vertices of an obtuse-angled triangle.'' Carroll
answered \(3\pi / (8\pi - 6\sqrt{3}) \approx 0.639\), but the derivation
conditions on the longest side in a way that quietly fixes a scale, and
the problem as posed has no answer at all: there is no uniform
distribution on the plane, and different limiting schemes give different
numbers. Portnoy (1994) works through Carroll's argument and the several
defensible alternatives; Guy (1993) and Eisenberg and Sullivan (1996)
reach the same conclusion from other directions. The situation is the
Bertrand paradox in a different costume, and the standard resolution is
the same one: state the sampling mechanism, then compute.

I have been thinking about this problem since I met Buffon's needle at
the end of an undergraduate degree in statistics. Buffon's needle is the
well-posed case: dropped on a ruled plane, its position and angle take
their distribution from the symmetry of the strips (M. G. Kendall and
Moran 1963, sec. 3.26). Carroll's three points have no such symmetry to
draw on, so the convention has to be supplied by hand.

One mechanism is unusually clean. Put the three points independently and
uniformly on a common circle. A triangle inscribed in a circle is acute
exactly when the center lies inside it, so the acute probability is the
probability that three uniform points on a circle contain the center,
which is \(1/4\) (Wendel 1962). The obtuse probability is therefore
\(3/4\), with no conditioning tricks and no limits.

A second route to the same number is more transparent still, and it will
matter later. Discard the points and keep the angles. The triple
\((A, B, C)\) with \(A + B + C = \pi\) ranges over an equilateral
triangle in barycentric coordinates, and Edelman and Strang (2015) take
the uniform distribution on it as the simplest model of a random
triangle. The angle at a given vertex is obtuse exactly when it exceeds
half the total, and that corner region is a copy of the whole simplex at
half scale, hence a quarter of the area. The three corners are disjoint,
so \(P(\text{obtuse}) = 3/4\), and the acute triangles form the medial
triangle, of area \(1/4\) (Figure~\ref{fig-angles}).

\begin{figure}

\centering{

\pandocbounded{\includegraphics[keepaspectratio]{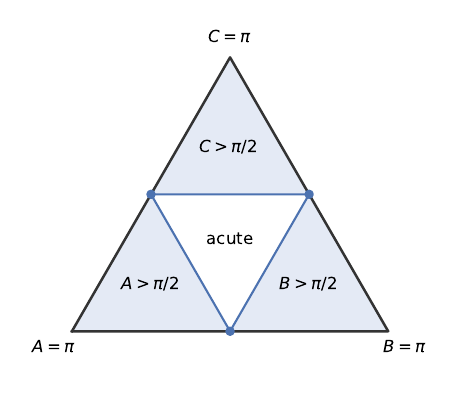}}

}

\caption{\label{fig-angles}Angle space. Every triangle is a point of
\(\{A + B + C = \pi\}\), drawn in barycentric coordinates, and the
uniform distribution on it is the simplest model of a random triangle.
The angle at \(A\) exceeds \(\pi/2\) in the shaded corner nearest the
vertex \(A = \pi\), and that corner is the whole picture at half scale,
so it holds a quarter of the area. The same goes for \(B\) and \(C\),
and the three cannot overlap, leaving the medial triangle, again a
quarter, as the acute triangles. Its edges are the right triangles and
its corners are degenerate, with one angle equal to zero.}

\end{figure}%

This is not a second convention but the first one in different
coordinates. Rooted at one of the sampled points, the three arc gaps of
three uniform points on a circle are uniform on
\(\{g_1 + g_2 + g_3 = 2\pi\}\), and the inscribed angle theorem makes
the interior angle at each vertex half the arc opposite it. The angle
triple is therefore uniform on \(\{A + B + C = \pi\}\): up to labelling
and reflection, the circular convention and the uniform-angle convention
induce the same law on triangle shape. (The rooting matters. Gaps read
off against a fixed external origin are size-biased and are not
Dirichlet.)

With Buffon's needle still in mind, I came at the question a third way
as an undergraduate, and this is the version that convinced me. Buffon
drops a needle of finite length \(\ell\); drop needles of infinite
length instead, which is to say lines, and three of them cut out a
triangle (Figure~\ref{fig-lines}). Toss them under the same invariant
measure that makes Buffon's problem well posed, \(dp\,d\theta\). Its
angular part is uniform, so take the directions
\(\theta_1, \theta_2, \theta_3\) independent and uniform on
\([0, \pi)\), and let the offsets be anything at all. The offset part is
not normalizable, exactly as in Carroll's problem, but here that costs
nothing: sliding a line parallel to itself moves and resizes the
triangle without changing any of its angles, so the shape depends on the
directions alone. Three lines fail to cut out a triangle only when two
of them are parallel or all three pass through a common point, and both
have probability zero. A right triangle needs an angle equal to
\(\pi/2\) exactly, which is also a null event, so obtuse and acute
between them account for everything.

Sort the directions as
\(\theta_{(1)} \leq \theta_{(2)} \leq \theta_{(3)}\) in \([0, \pi)\).
The interior angles are then \[
\theta_{(2)} - \theta_{(1)}, \qquad \theta_{(3)} - \theta_{(2)}, \qquad
\pi - \left(\theta_{(3)} - \theta_{(1)}\right) ,
\] which are the three gaps between the directions on a circle of
circumference \(\pi\). A triangle is obtuse exactly when one angle
exceeds \(\pi/2\), that is when one gap covers more than half of that
circle. Fix a direction and look at the gap running counterclockwise
from it: that gap covers more than half the circle precisely when the
other two directions both fall in the opposite half, which has
probability \(\left(\tfrac12\right)^2 = \tfrac14\). Two gaps cannot both
exceed half, so the three cases are disjoint and \[
P(\text{obtuse}) = 3 \times \tfrac14 = \tfrac34 .
\]

Three constructions, then, and one measure. Uniform points on a circle,
the uniform distribution on angle space, and lines under the invariant
measure all induce the same law on the shape of a triangle.

\begin{figure}

\centering{

\pandocbounded{\includegraphics[keepaspectratio]{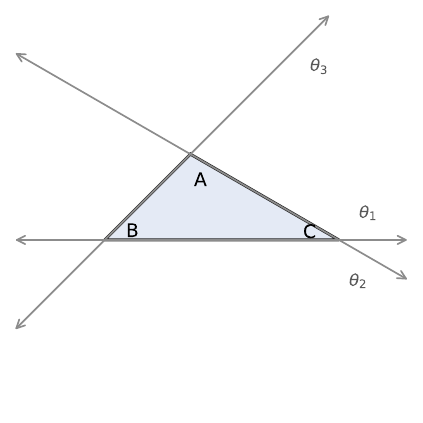}}

}

\caption{\label{fig-lines}Three needles of infinite length, which is to
say lines, tossed onto the plane and drawn with arrowheads to mark that
they run on forever. They cut out a triangle. Its interior angles \(A\),
\(B\), \(C\) are the gaps between the line directions
\(\theta_1, \theta_2, \theta_3\), so sliding any line parallel to itself
moves and resizes the triangle but leaves all three angles alone.}

\end{figure}%

A genuinely different convention reaches the same \(3/4\) by an argument
worth seeing on its own terms. Take the six vertex coordinates to be
independent standard Gaussians. Shape theory then places the triangle
uniformly on a hemisphere of shapes, a result of David Kendall and
collaborators (D. G. Kendall 1989), who is not the M. G. Kendall of the
paragraph above, and Edelman and Strang (2015) show (their Theorem 7)
that under that law the squared side lengths, normalized so that
\(a^2 + b^2 + c^2 = 1\), are each uniform on \([0, 2/3]\). The angle
opposite \(c\) is obtuse exactly when \(c^2 > a^2 + b^2\), which under
the normalization is simply \(c^2 > 1/2\). So a given angle is obtuse
with probability \[
\frac{2/3 - 1/2}{2/3} = \frac{1}{4} ,
\] and because at most one angle of a triangle can be obtuse, the three
cases add to \(3/4\).

Both moves in that argument reappear below. Additivity over the three
vertices is Lemma~\ref{lem-additive}. A Pythagorean threshold on squared
quantities decides which vertex can carry the obtuse angle in
Corollary~\ref{cor-support}, though there the inequality constrains the
radii of the circles rather than the sides of the triangle.

This paper takes that formulation and gives each vertex its own circle.
Fix radii \(r_1, r_2, r_3 \geq 0\) and let vertex \(k\) be uniform on
the circle of radius \(r_k\), all three circles concentric. Setting
\(r_1 = r_2 = r_3\) must return \(3/4\), and it does. The general
answer, Theorem~\ref{thm-main} below, is a sum of three terms, each of
which is the probability that a weighted sum of two independent arcsine
random variables clears a threshold. The formula is a single
one-dimensional integral to evaluate, it reduces to closed forms in
several subfamilies, and it makes the qualitative structure visible:
which vertex can carry the obtuse angle, how low the obtuse probability
can go, and how it behaves as the radii spread apart.

Random triangles have a large literature, including shape-space and
random matrix treatments (Edelman and Strang 2015). The
concentric-circle family is a two-parameter deformation of the classical
circular case that stays elementary throughout: the tools are the law of
sines and a change of variables.

\section{Setup}\label{sec-setup}

Write \(O\) for the common center of the three circles and let
\(r_k \geq 0\) denote the radius of the circle carrying vertex \(k\),
with at most one of the three radii equal to zero. (Two zero radii would
place two vertices at \(O\) and there would be no triangle.) Let
\(\Theta_1, \Theta_2, \Theta_3\) be independent and uniform on
\([0, 2\pi)\) and set \[
P_k = r_k\left(\cos \Theta_k,\ \sin \Theta_k\right), \qquad k = 1, 2, 3 ,
\] so \(P_k\) is uniform on the circle of radius \(r_k\) about \(O\).
Throughout, \(\{i, j, k\}\) denotes \(\{1, 2, 3\}\) in some order.
Scaling all three radii by a common factor scales the triangle without
changing any of its angles, so only the ratios \(r_1 : r_2 : r_3\)
matter.

Let \(p_k\) be the probability that the interior angle at \(P_k\) is
obtuse, and let
\(P(\text{obtuse}) = \Pr(\text{the triangle } P_1P_2P_3 \text{ is obtuse})\).

\begin{lemma}[Additivity]\protect\hypertarget{lem-additive}{}\label{lem-additive}

\(P(\text{obtuse}) = p_1 + p_2 + p_3\).

\end{lemma}

\emph{Proof.} The three interior angles sum to \(\pi\), so at most one
of them exceeds \(\pi/2\) and the three events are disjoint. Their union
is the event that the triangle is obtuse, up to the event that some
angle equals exactly \(\pi/2\). The angle at \(P_k\) is right when \[
D_k = (P_i - P_k) \cdot (P_j - P_k) = r_ir_j\cos(\Theta_i - \Theta_j)
  - r_ir_k\cos(\Theta_i - \Theta_k) - r_jr_k\cos(\Theta_j - \Theta_k) + r_k^2
\] vanishes. As a function on the torus \(D_k\) is real analytic, and it
is not identically zero. Its mean over the torus is \(r_k^2\), since
each of the three cosine terms integrates to zero, and that is positive
unless \(r_k = 0\); when \(r_k = 0\) it reduces to
\(r_ir_j\cos(\Theta_i - \Theta_j)\), which is not identically zero
because at most one radius vanishes. The zero set of a nonvanishing real
analytic function has measure zero. The same argument disposes of the
collinear and coincident configurations, on which ``obtuse'' is
undefined: both are cut out by real analytic equations that do not
vanish identically. \(\square\)

Lemma~\ref{lem-additive} is what makes the problem tractable. It
converts a question about a triangle into three separate questions about
a single vertex, each of which involves only two independent angles.

\section{The angle at one vertex}\label{sec-vertex}

Fix the vertex \(P_k\) and one companion \(P_i\). The angle at \(P_k\)
is obtuse exactly when \((P_i - P_k) \cdot (P_j - P_k) < 0\), that is,
when \(P_j\) falls in the open half plane \(H\) bounded by the line
through \(P_k\) perpendicular to \(P_kP_i\), on the side away from
\(P_i\) (Figure~\ref{fig-geometry}). Since \(P_j\) is uniform on its own
circle, the conditional probability is just the fraction of that circle
lying in \(H\), and that fraction depends on \(P_k\) and \(P_i\) only
through the signed distance from the center to the boundary line.

\begin{figure}

\centering{

\pandocbounded{\includegraphics[keepaspectratio]{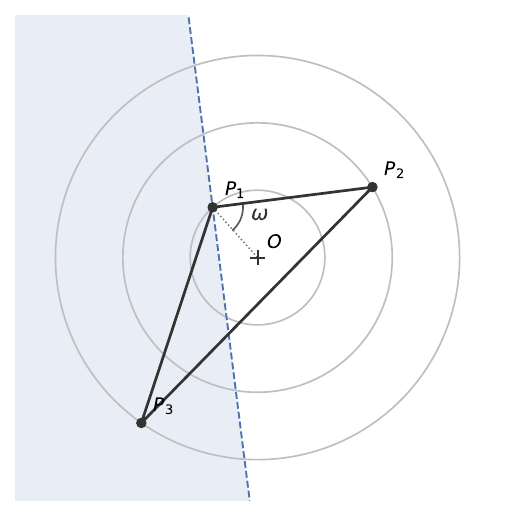}}

}

\caption{\label{fig-geometry}One realization with radii \((1, 2, 3)\).
The angle at \(P_1\) is obtuse exactly when \(P_3\) falls in the shaded
half plane, whose boundary passes through \(P_1\) perpendicular to
\(P_1P_2\).}

\end{figure}%

Let \[
\omega = \angle O P_k P_i \in [0, \pi]
\] be the angle at \(P_k\) in the triangle \(OP_kP_i\); it is the angle
subtended at \(P_k\) between the center and the companion vertex. The
unit normal of \(H\) pointing away from \(P_i\) is
\(n = -(P_i - P_k)/\lVert P_i - P_k \rVert\), and because \(\omega\) is
measured between the directions \(P_k \to O\) and \(P_k \to P_i\), we
have \(P_k \cdot n = r_k \cos\omega\). So \(H\) contains the arc of the
circle of radius \(r_j\) on which \(r_j \cos \Xi < -r_k\cos\omega\) for
\(\Xi\) uniform, giving

\begin{equation}\phantomsection\label{eq-conditional}{
\Pr(\text{angle at } P_k \text{ is obtuse} \mid P_i)
= \frac{1}{\pi}\arccos\left(\left[\frac{r_k\cos\omega}{r_j}\right]\right)
= \frac{1}{2} - \frac{1}{\pi}
  \arcsin\left(\left[\frac{r_k\cos\omega}{r_j}\right]\right) ,
}\end{equation}

where \([x] = \max(-1, \min(1, x))\) truncates to the domain of
\(\arcsin\); the truncation is what happens when the circle of radius
\(r_j\) lies entirely on one side of the line.

What remains is the distribution of \(\omega\).

\begin{lemma}[The viewing
angle]\protect\hypertarget{lem-omega}{}\label{lem-omega}

Let \(\psi = \angle O P_i P_k\) and let
\(h = r_k \sin\omega = r_i \sin\psi\) be the distance from the center to
the line \(P_kP_i\). Then

\begin{itemize}
\tightlist
\item
  if \(r_k \leq r_i\), the folded angle \(\min(\omega, \pi - \omega)\)
  is uniform on \((0, \pi/2)\);
\item
  if \(r_k \geq r_i\), the acute branch
  \(\psi_0 = \min(\psi, \pi - \psi)\) is uniform on \((0, \pi/2)\). (The
  line meets the circle twice, and \(\psi\) is obtuse at the far point,
  so \(\psi\) itself is not confined to \((0, \pi/2)\).)
\end{itemize}

In both cases \(h = \min(r_i, r_k)\sin\mu\) for a \(\mu\) uniform on
\((0, \pi/2)\).

\end{lemma}

\emph{Proof.} Let \(A = \angle P_k O P_i\) be the central angle. Since
\(\Theta_i - \Theta_k\) is uniform on the circle, \(A\) is uniform on
\((0, \pi)\). The interior angles satisfy \(A + \omega + \psi = \pi\),
and the law of sines gives \(r_i \sin\psi = r_k \sin\omega\), so
\(\sin\psi = (r_k/r_i)\sin\omega\).

If \(r_k \leq r_i\) then \(\psi\) is opposite the shorter side and is
therefore acute, so \(\psi = \arcsin\left((r_k/r_i)\sin\omega\right)\)
is a function of \(\omega\) alone and
\(\omega = \pi - A - \psi(\omega)\). Hence
\begin{equation}\phantomsection\label{eq-cdf-inside}{
\Pr(\omega \leq w) = \Pr\left(A \geq \pi - w - \psi(w)\right)
  = \frac{w + \psi(w)}{\pi}, \qquad 0 \leq w \leq \pi .
}\end{equation} Both \(\omega \leq w\) and \(\omega \geq \pi - w\)
contribute to \(\min(\omega, \pi - \omega) \leq w\) for
\(w \leq \pi/2\), and adding the two pieces from
Equation~\ref{eq-cdf-inside} gives \(2w/\pi\), because
\(\psi(w) = \psi(\pi - w)\).

If \(r_k \geq r_i\) the line from \(P_k\) meets the circle of radius
\(r_i\) twice, at the near point (where \(\psi\) is acute) and the far
point (where \(\psi\) is obtuse), and
\(\omega \leq \arcsin(r_i/r_k) \leq \pi/2\). Writing \(\psi_0\) for the
acute value, the two branches contribute
\(\Pr(A \geq \pi - \omega - \psi_0)\) and
\(\Pr(A \leq \psi_0 - \omega)\), which sum to \(2\psi_0/\pi\).
\(\square\)

Lemma~\ref{lem-omega} has a consequence beyond its use here: the
distance from the center to the chord through a fixed point and a
uniform point on a circle has the law \(\min(r_i, r_k) \sin \mu\) with
\(\mu\) uniform, and so depends on the two radii only through their
minimum. I record that in passing and claim no priority for it; random
chords are old ground in integral geometry and I have not searched that
literature. Evaluating Equation~\ref{eq-conditional} needs slightly more
than \(h\), however, since it needs the sign of \(\cos\omega\) as well.
That is the content of the next lemma, which does the bookkeeping once
so the theorem can be read off.

\begin{lemma}[Signed
transfer]\protect\hypertarget{lem-signed}{}\label{lem-signed}

Let \(S = \operatorname{sign}(\cos\omega)\). For every bounded
measurable \(\phi\), \begin{equation}\phantomsection\label{eq-transfer}{
\mathbb{E}\left[S\,\phi(h)\right]
  = \frac{2}{\pi}\int_0^{\pi/2}
    \phi\left(r_i \sin t\right)\,
    \mathbf{1}\!\left\{r_i \sin t < r_k\right\} dt .
}\end{equation}

\end{lemma}

\emph{Proof.} Suppose first \(r_k \geq r_i\). Then
\(\omega \leq \pi/2\), so \(S \equiv 1\), and
\(h = r_i\sin\psi \leq r_i \leq r_k\) with \(\psi\) uniform on
\((0, \pi/2)\) by Lemma~\ref{lem-omega}; the indicator holds almost
surely and Equation~\ref{eq-transfer} is immediate.

Suppose instead \(r_k < r_i\) and write \(G\) for the distribution
function Equation~\ref{eq-cdf-inside}, so that
\(\pi \, dG = d\omega + d\psi\). Split the integral accordingly. The
\(d\omega\) part is \[
\frac{1}{\pi}\int_0^{\pi} \operatorname{sign}(\cos\omega)\,
  \phi\!\left(r_k\sin\omega\right) d\omega = 0 ,
\] because that component has constant density \(1/\pi\) and so is
preserved by \(\omega \mapsto \pi - \omega\), an involution that
reverses the sign of \(\cos\omega\) while fixing \(\sin \omega\). In the
\(d\psi\) part, \(\psi\) increases from \(0\) to \(\arcsin(r_k/r_i)\) as
\(\omega\) runs over \((0, \pi/2)\), where \(S = +1\), and decreases
back to \(0\) on \((\pi/2, \pi)\), where \(S = -1\); reversing the
orientation of the second piece shows that the two contribute equally.
Since \(r_k \sin\omega = r_i\sin\psi\), \[
\mathbb{E}\left[S\,\phi(h)\right]
 = \frac{2}{\pi}\int_0^{\arcsin(r_k/r_i)} \phi\left(r_i\sin t\right) dt ,
\] and \(r_i \sin t < r_k\) is exactly \(t < \arcsin(r_k/r_i)\).
\(\square\)

\section{The main result}\label{sec-main}

\begin{theorem}[Obtuse probability on concentric
circles]\protect\hypertarget{thm-main}{}\label{thm-main}

Let \(X\) and \(Y\) be independent and uniform on \((0, \pi/2)\). Then
\begin{equation}\phantomsection\label{eq-pk}{
p_k = \frac{1}{2}\,
  \Pr\!\left(r_i^2 \sin^2 X + r_j^2 \sin^2 Y > r_k^2\right) ,
}\end{equation} and consequently
\begin{equation}\phantomsection\label{eq-main}{
P(\text{obtuse}) = \frac{1}{2}\sum_{k=1}^{3}
  \Pr\!\left(r_i^2 \sin^2 X + r_j^2 \sin^2 Y > r_k^2\right) .
}\end{equation}

\end{theorem}

\emph{Proof.} Take expectations in Equation~\ref{eq-conditional}. Since
\(r_k\cos\omega = S\sqrt{r_k^2 - h^2}\) and \(\arcsin\) is odd, the
truncated arcsine factors as \(S\,\phi(h)\) with \[
\phi(h) = \arcsin\left[\min\left(1,\
  \frac{\sqrt{\max(r_k^2 - h^2,\, 0)}}{r_j}\right)\right] .
\] Lemma~\ref{lem-signed} then gives
\begin{equation}\phantomsection\label{eq-quadrature}{
p_k = \frac{1}{2} - \frac{2}{\pi^2}\int_0^{\pi/2}
  \arcsin\left[\min\left(1,\
  \frac{\sqrt{r_k^2 - r_i^2\sin^2 t}}{r_j}\right)\right]
  \mathbf{1}\!\left\{r_i\sin t < r_k\right\} dt .
}\end{equation} For fixed \(t\) with \(r_i \sin t < r_k\), \[
\frac{2}{\pi}\arcsin\left[\min\left(1,\
  \frac{\sqrt{r_k^2 - r_i^2\sin^2 t}}{r_j}\right)\right]
 = \Pr\!\left(r_j \sin Y < \sqrt{r_k^2 - r_i^2 \sin^2 t}\right)
 = \Pr\!\left(r_i^2\sin^2 t + r_j^2 \sin^2 Y < r_k^2\right) ,
\] and when \(r_i \sin t \geq r_k\) that probability is zero, which
absorbs the indicator. Integrating over \(t\) against the uniform law of
\(X\) turns Equation~\ref{eq-quadrature} into
\(p_k = \tfrac12 - \tfrac12\Pr(r_i^2\sin^2 X +
r_j^2\sin^2 Y < r_k^2)\), which is Equation~\ref{eq-pk}.
Lemma~\ref{lem-additive} gives Equation~\ref{eq-main}. \(\square\)

First, Equation~\ref{eq-pk} is symmetric in \(i\) and \(j\), as it must
be, even though the derivation broke that symmetry by conditioning on
\(P_i\).

Second, \(X\) and \(Y\) are auxiliary. They are not the two viewing
angles of the same realized triangle, and the corresponding event is not
a pointwise restatement of ``the angle at \(P_k\) is obtuse.'' It is an
identity between probabilities, not between events. Equation~\ref{eq-pk}
is reached by integrating the companion vertex out through
Lemma~\ref{lem-signed} rather than by any pathwise coupling, so there is
no realization in which \(X\) and \(Y\) are the two viewing angles.

Third, if \(U = \sin^2 X\) then \(U\) has the arcsine distribution on
\([0,1]\), with density \(1 / (\pi\sqrt{u(1-u)})\), so
Equation~\ref{eq-main} can be written with \(U\) and \(V\) independent
and identically distributed (iid),
\begin{equation}\phantomsection\label{eq-arcsine}{
P(\text{obtuse}) = \frac{1}{2}\sum_{k=1}^{3}
  \Pr\!\left(r_i^2 U + r_j^2 V > r_k^2\right), \qquad
U, V \overset{\text{iid}}{\sim} \text{Arcsine}(0, 1) .
}\end{equation} The arcsine law is the standard companion of uniform
angles (M. G. Kendall and Moran 1963), and in this form the three terms
are visibly the same function of the squared radii evaluated at three
different thresholds.

For computation, Equation~\ref{eq-quadrature} is a one-dimensional
integral of an elementary function with two kinks, at
\(t = \arcsin\min(1, r_k/r_i)\) and where the arcsine saturates.
Adaptive quadrature with those break points supplied is accurate to a
few units in the last place: it reproduces the two closed forms of
Section~\ref{sec-cases} to within \(4 \times 10^{-16}\), and its own
reported error estimate stays below \(3 \times 10^{-9}\) over a grid of
radii.

\section{Special cases}\label{sec-cases}

\begin{corollary}[Equal
radii]\protect\hypertarget{cor-equal}{}\label{cor-equal}

If \(r_1 = r_2 = r_3\) then \(p_k = 1/4\) for each \(k\) and
\(P(\text{obtuse}) = 3/4\).

\end{corollary}

\emph{Proof.} With a common radius the event in Equation~\ref{eq-pk} is
\(\sin^2 X + \sin^2 Y > 1\), that is \(\cos^2 X < \sin^2 Y\), that is
\(X + Y > \pi/2\), which has probability \(1/2\) by symmetry.
\(\square\)

This is the classical answer, recovered without appealing to the
inscribed-angle argument. The two routes agree: for a common circle,
acute is the same as containing the center, and three uniform points
contain the center with probability \(1/4\) (Wendel 1962). That second
route does not survive the generalization, and it repays being precise
about which half of it fails.

\begin{proposition}[The center and the circumcenter part
company]\protect\hypertarget{prp-center}{}\label{prp-center}

If \(r_1, r_2, r_3 > 0\) then
\(\Pr\left(O \in \operatorname{conv}\{P_1, P_2, P_3\}\right) = 1/4\),
whatever the radii. Moreover, with the positivity hypothesis dropped,
\(\Pr(\text{acute}) = 1 - P(\text{obtuse})\) takes every value in
\((0, 1/2]\) as the radii range over all admissible triples.

\end{proposition}

\emph{Proof.} Because \(r_k > 0\), the point \(P_k\) and the unit vector
\(u_k = P_k/r_k\) lie on the same open ray from \(O\), hence strictly on
the same side of every line through \(O\). So \(O\) lies in the convex
hull of \(P_1, P_2, P_3\) if and only if \(u_1, u_2, u_3\) fail to lie
in a common open half plane, an event determined by
\(\Theta_1, \Theta_2, \Theta_3\) alone. Those angles are uniform
whatever the radii, and Wendel's theorem (Wendel 1962) gives
\(2^{-2}\left[\binom{2}{0} + \binom{2}{1}\right] = 3/4\) for the
probability that three uniform directions in the plane do lie in a half
plane, leaving \(1/4\) for the containment. The second claim carries no
positivity hypothesis, and it needs none: Proposition~\ref{prp-bounds}
below confines \(\Pr(\text{acute})\) to \((0, 1/2]\), and along the
slice \((1, 1, \rho)\) the obtuse probability is a continuous function
of \(\rho\) that equals \(1/2\) at \(\rho = 0\) and tends to \(1\) as
\(\rho \to \infty\), so every intermediate value is attained.
\(\square\)

The endpoint \(\rho = 0\) of that slice is precisely the configuration
the first claim excludes. With a vertex sitting at \(O\) the containment
probability jumps from \(1/4\) to \(1\), so the two halves of the
proposition are describing genuinely different regimes rather than the
same one twice.

A triangle is acute exactly when it contains its circumcenter, and the
circumcenter is the common center of the circles only when the three
radii agree. Unequal radii pull those two points apart, and the two
probabilities follow suit: containment is pinned at \(1/4\) by the
directions alone, while acuteness responds to the radii and ranges over
the whole of \((0, 1/2]\). Equation~\ref{eq-main} is the measurement of
how far the two have separated. They agree numerically on the level set
\(P(\text{obtuse}) = 3/4\), which is a curve rather than a point, but
the underlying events coincide only at equal radii.

\begin{corollary}[Which vertex can be
obtuse]\protect\hypertarget{cor-support}{}\label{cor-support}

\(p_k > 0\) if and only if \(r_k^2 < r_i^2 + r_j^2\), and \(p_k = 1/2\)
if and only if \(r_k = 0\). At most one of the three inequalities can
fail, except when one radius is zero and the other two are equal, in
which case two fail.

\end{corollary}

\emph{Proof.} The random variable \(r_i^2\sin^2 X + r_j^2\sin^2 Y\) has
support \([0, r_i^2 + r_j^2]\) with positive density in the interior, so
the probability in Equation~\ref{eq-pk} is positive exactly when
\(r_k^2 < r_i^2 + r_j^2\) and equals one exactly when \(r_k = 0\). If
two of the inequalities failed, say \(r_k^2 \geq r_i^2 + r_j^2\) and
\(r_j^2 \geq r_i^2 + r_k^2\), adding them forces \(r_i = 0\) and then
\(r_j = r_k\). \(\square\)

So the vertex sitting on a circle that dominates the other two in the
Pythagorean sense is never the obtuse one, and the sum in
Equation~\ref{eq-main} has only two nonzero terms whenever the radii are
spread out. Being the largest is not enough on its own, though: at radii
\((1, 1, 1.2)\) the outer vertex still carries the obtuse angle with
probability 0.1058.

\begin{corollary}[One vertex at the
center]\protect\hypertarget{cor-center}{}\label{cor-center}

If \(r_3 = 0\) and \(r_1 \leq r_2\) then \[
P(\text{obtuse}) = 1 - \frac{1}{\pi}\arcsin\!\left(\frac{r_1}{r_2}\right) .
\]

\end{corollary}

\emph{Proof.} By Corollary~\ref{cor-support}, \(p_3 = 1/2\), and
\(p_2 = 0\) because \(r_2^2 \geq r_1^2 + 0\). For the remaining term,
\(p_1 = \tfrac12\Pr(r_2^2\sin^2 X > r_1^2) = \tfrac12\left(1 -
\tfrac{2}{\pi}\arcsin(r_1/r_2)\right)\). \(\square\)

The expression decreases from \(1\) at \(r_1/r_2 \to 0\) to \(1/2\) at
\(r_1 = r_2\). The lower end is the configuration of
Proposition~\ref{prp-bounds} below.

\begin{corollary}[Two equal
radii]\protect\hypertarget{cor-chi}{}\label{cor-chi}

If \(r_1 = r_2 = 1\) and \(r_3 = \rho \geq \sqrt{2}\) then \[
P(\text{obtuse}) = 1 - \frac{4}{\pi^2}\,\chi_2\!\left(\frac{1}{\rho}\right),
\qquad
\chi_2(s) = \sum_{n \geq 0} \frac{s^{2n+1}}{(2n+1)^2} ,
\] where \(\chi_2\) is the Legendre chi function.

\end{corollary}

\emph{Proof.} Corollary~\ref{cor-support} gives \(p_3 = 0\). For \(p_1\)
(and by symmetry \(p_2\)), Equation~\ref{eq-pk} asks for
\(\Pr(\sin^2 X + \rho^2\sin^2 Y > 1) =
\Pr(\rho \sin Y > \cos X)\), and conditioning on \(X\) gives \[
p_1 = \frac{1}{2}\left(1 - \frac{4}{\pi^2}
  \int_0^{\pi/2}\arcsin\left(\frac{\cos x}{\rho}\right) dx\right) .
\] Expanding \(\arcsin\) in its Taylor series and integrating term by
term with Wallis' formula,
\(\int_0^{\pi/2}\cos^{2n+1}x \, dx = (2n)!! / (2n+1)!!\), collapses the
coefficients to \(\int_0^{\pi/2}\arcsin(s\cos x)\,dx = \chi_2(s)\).
\(\square\)

The per-vertex expression for \(p_1\) holds on the wider range
\(\rho \geq 1\), which is where \(\arcsin(\cos x/\rho)\) stays clear of
its truncation; the hypothesis \(\rho \geq \sqrt{2}\) additionally kills
the third term, and is therefore the single condition under which the
displayed total is valid. On the wider range \(\chi_2(1) = \pi^2/8\)
returns \(p_1 = 1/4\) at \(\rho = 1\), consistent with
Corollary~\ref{cor-equal}. Parity is what puts \(\chi_2\) here rather
than the dilogarithm itself:
\(\chi_2(s) = \tfrac12\left[\mathrm{Li}_2(s) - \mathrm{Li}_2(-s)\right]\)
is the odd part of \(\mathrm{Li}_2\), and \(\arcsin\) being an odd
function leaves only odd powers of \(s\) to integrate.

\begin{proposition}[Bounds]\protect\hypertarget{prp-bounds}{}\label{prp-bounds}

\(1/2 \leq P(\text{obtuse}) < 1\). The lower bound is attained exactly
when one radius is zero and the other two are equal, and
\(P(\text{obtuse}) \to 1\) as the radii separate.

\end{proposition}

\emph{Proof.} Write \(a \leq b \leq c\) for the sorted values of
\(r_1^2, r_2^2, r_3^2\), let \(U, V\) be independent arcsine variables
as in Equation~\ref{eq-arcsine}, and let \(q_a = \Pr(bU + cV > a)\),
\(q_b = \Pr(aU + cV > b)\), and \(q_c = \Pr(aU + bV > c)\) be the
corresponding terms of Equation~\ref{eq-arcsine}, so that
\(2P(\text{obtuse}) = q_a + q_b + q_c\).

The arcsine law is symmetric about \(1/2\), so \((U, V)\) and
\((1 - U, 1 - V)\) have the same distribution and \[
q_b = \Pr\left(a(1 - U) + c(1 - V) > b\right) = \Pr(aU + cV < a + c - b) .
\] Let \(E_1 = \{bU + cV \leq a\}\) and
\(E_2 = \{aU + cV < a + c - b\}\), so that
\(q_a + q_b = 1 - \Pr(E_1) + \Pr(E_2)\). On \(E_1\) we have
\(aU + cV \leq bU + cV \leq a \leq a + c - b\), with at least one
inequality strict unless \(a = b = c\); hence \(E_1 \subseteq E_2\) up
to a null set and \(q_a + q_b \geq 1\). Dropping \(q_c \geq 0\) gives
\(P(\text{obtuse}) \geq 1/2\). (When \(a = b = c\) the three terms are
each \(1/2\) and the bound is not tight.)

Equality forces \(q_c = 0\) and \(\Pr(E_2 \setminus E_1) = 0\). Suppose
\(a > 0\). If \(c > b\), take \(U\) close to \(1\) and \(V\) close to
\(0\) subject to \(cV > a(1 - U)\). Then
\(bU + cV > bU + a(1 - U) \geq a\), so the point is outside \(E_1\),
while \(aU + cV\) is close to \(a\), which is strictly below
\(a + c - b\), so it lies in \(E_2\). If instead \(c = b > a\), reverse
the constraint to \(cV < a(1 - U)\). Then
\(aU + cV < aU + a(1 - U) = a = a + c - b\) puts the point in \(E_2\),
while \(bU + cV\) is close to \(c > a\) and so is outside \(E_1\). Both
prescriptions describe sets of positive Lebesgue measure, on which the
arcsine density is positive, so \(\Pr(E_2 \setminus E_1) > 0\) in either
case.

That leaves \(a = b = c\) as the only configuration with \(a > 0\) not
yet excluded, and there \(q_c = 1/2\), so equality fails through the
dropped term instead. Hence \(a = 0\).

It is tempting to finish by saying that \(q_c = 0\) then forces
\(b = c\), but \(q_c = \Pr(bV > c)\) vanishes for \emph{every}
\(b \leq c\) once \(a = 0\), so it carries no information. The other
equality condition does the work. At \(a = 0\) the set
\(E_1 = \{bU + cV \leq 0\}\) is null, while \(E_2 = \{cV < c - b\}\) has
positive probability whenever \(b < c\); so
\(\Pr(E_2\setminus E_1) = 0\) forces \(b \geq c\), and with \(b \leq c\)
that gives \(b = c\). At \(a = 0\) with \(b = c\) we have \(q_a = 1\),
\(q_b = 0\), and \(q_c = 0\).

For the upper bound, fix the radii and work in the angle torus, where
the law is uniform. The three dot products are continuous in
\((\Theta_1, \Theta_2, \Theta_3)\), so the acute configurations form an
open subset of the torus, and exhibiting a single one gives it positive
probability. Order the radii so that \(r_3 = \max\).

If all three are positive and not all equal, place \[
P_1 = (0, r_1), \qquad P_2 = (0, -r_2), \qquad P_3 = (r_3, 0) .
\] The three vertex dot products are \(r_1(r_1 + r_2)\),
\(r_2(r_1 + r_2)\), and \(r_3^2 - r_1r_2\). The first two are positive,
and the third is positive because \(r_3 = \max\) gives
\(r_3^2 \geq r_1r_2\) with equality only when \(r_1 = r_2 = r_3\). If
instead all three are equal, three points \(2\pi/3\) apart give dot
products \(3r^2/2 > 0\). If one radius vanishes, say
\(r_1 = 0 < r_2 \leq r_3\), put the other two at angle \(\theta\) with
\(0 < \cos\theta < r_2/r_3\); the dot products are then
\(r_2r_3\cos\theta\), \(r_2(r_2 - r_3\cos\theta)\), and
\(r_3(r_3 - r_2\cos\theta)\), all positive. So \(P(\text{obtuse}) < 1\)
in every case, and Corollary~\ref{cor-center} shows values arbitrarily
close to \(1\), as \(r_1/r_2 \to 0\) with \(r_3 = 0\). \(\square\)

Name the extremal configuration: one vertex pinned at the center and the
other two on a common circle. There the triangle is isosceles, both base
angles are acute by construction, and the triangle is obtuse exactly
when the central angle exceeds \(\pi/2\), which happens half the time.
No arrangement of concentric circles does better than an even split.

At the other end of the range sits the equal-radius configuration, which
is a local maximum but not a smooth one.

\begin{proposition}[Local maximum at equal
radii]\protect\hypertarget{prp-local}{}\label{prp-local}

Normalize the squared radii to \(a_k = r_k^2\) with
\(a_1 + a_2 + a_3 = 3\) and write \(a_k = 1 + \alpha_k\). Then, as
\(\alpha \to 0\), \begin{equation}\phantomsection\label{eq-local}{
P(\text{obtuse}) = \frac{3}{4} - \frac{1}{4\pi^2}\sum_{k<l}
  (\alpha_k - \alpha_l)^2 \log\frac{1}{\lvert \alpha_k - \alpha_l \rvert}
  + O\!\left(\lvert\alpha\rvert^2\right) ,
}\end{equation} where a vanishing difference contributes nothing to the
sum. In particular \(3/4\) is a strict local maximum over radius ratios,
attained only at \(r_1 = r_2 = r_3\), and \(P\) is differentiable but
not twice differentiable there. The remainder is uniform in the
direction of \(\alpha\); Lemma~\ref{lem-endpoint} supplies the exact
endpoint decomposition on which that rests.

\end{proposition}

\emph{Proof.} Condition on \(X\) in Equation~\ref{eq-pk} and use
\(\Pr(\sin^2 Y \leq z) = \tfrac{2}{\pi}\arcsin\sqrt{z}\) to write
\(q_k = 2p_k\) as a single integral,
\begin{equation}\phantomsection\label{eq-arccos}{
q_k = 1 - \frac{4}{\pi^2} I_k , \qquad
I_k = \int_0^{\pi/2}
  \arccos\left(\min\left(1,\
  \sqrt{\left[\mu \sin^2 x - \eta\right]^{+}}\right)\right) dx ,
}\end{equation} with \(\mu = a_i/a_j\) and \(\eta = (a_k - a_j)/a_j\),
and \([\cdot]^{+}\) the positive part. At \(\alpha = 0\) the integrand
is \(\arccos(\sin x) = \pi/2 - x\) and \(I_k = \pi^2/8\).

The integrand of Equation~\ref{eq-arccos} is a smooth function of
\((\alpha, x)\) on any interval \([\delta, \pi/2 - \delta]\), so only
the two endpoints can produce behavior worse than
\(O(\lvert\alpha\rvert)\). Near \(x = 0\) both arguments of \(\arccos\)
are small, \(\arccos t - \arccos t_0 = t_0 - t + O(t^3)\), and
\(\sin x = x + O(x^3)\), so that endpoint contributes \[
\int_0^{\delta}\left(x - \sqrt{\left[\mu x^2 - \eta\right]^{+}}\right) dx
 = \frac{\eta}{4\sqrt{\mu}}\log\frac{1}{\lvert\eta\rvert}
   + O(\lvert\alpha\rvert)
 = \frac{\eta}{4}\left(1 - \frac{\lambda}{2}\right)
   \log\frac{1}{\lvert\eta\rvert} + O(\lvert\alpha\rvert) ,
\] where \(\lambda = \mu - 1 = (a_i - a_j)/a_j\); the logarithm appears
because the integrand is cut off at \(x \asymp \sqrt{\lvert\eta\rvert}\)
rather than at \(\lvert\eta\rvert\). Substituting
\(x \mapsto \pi/2 - x\) gives the same expression at the other endpoint
with \(i\) and \(j\) interchanged.

Now sum over \(k\) and over the two endpoints, which is a sum over
ordered pairs. Writing \(e_{kl} = \alpha_k - \alpha_l\), we have
\(\eta = e_{kj}(1 + O(\lvert\alpha\rvert))\) and
\(\lambda = e_{ij}(1 + O(\lvert\alpha\rvert))\), so the ordered pair
\((k, j)\) contributes
\(\tfrac14\log(1/\lvert e_{kj}\rvert)\left[e_{kj} - e_{kj}\alpha_j -
\tfrac12 e_{kj}e_{ij}\right]\) up to \(O(\lvert\alpha\rvert^2)\). Pair
\((k, j)\) with \((j, k)\), which share the logarithm and the third
index \(i\). Because \(e_{jk} = -e_{kj}\), the terms linear in \(e\)
cancel; the second terms combine to
\(-e_{kj}\alpha_j - e_{jk}\alpha_k = e_{kj}(\alpha_k - \alpha_j) = e_{kj}^2\);
and the third terms combine to
\(-\tfrac12\left(e_{kj}e_{ij} + e_{jk}e_{ik}\right)
= \tfrac12 e_{kj}(e_{ik} - e_{ij}) = -\tfrac12 e_{kj}^2\). Hence \[
\sum_{k=1}^{3} I_k = \frac{3\pi^2}{8}
  + \frac{1}{8}\sum_{k<l} e_{kl}^2 \log\frac{1}{\lvert e_{kl}\rvert}
  + \Lambda(\alpha) + O(\lvert\alpha\rvert^2) ,
\] where \(\Lambda\) collects the smooth contributions and is
differentiable at \(\alpha = 0\). Since \(P\) is symmetric in the three
radii, \(\nabla\Lambda(0)\) is parallel to \((1, 1, 1)\), and \(P\) is
constant in that direction by scale invariance, so \(\Lambda\) has no
linear term on the normalized slice \(\sum_k \alpha_k = 0\). Finally
\(2P(\text{obtuse}) = \sum_k q_k =
3 - \tfrac{4}{\pi^2}\sum_k I_k\), which is Equation~\ref{eq-local}.

Because \(t^2\log(1/t)\) dominates \(t^2\) as \(t \downarrow 0\), the
displayed sum controls the remainder, and it is strictly positive unless
all three \(\alpha_k\) agree. The same domination shows the second
directional derivative diverges in every direction other than the
dilation. \(\square\)

Two steps of that proof were quoted rather than carried out: the
uniformity of the remainder in the direction of \(\alpha\), and the
differentiability of the smooth part \(\Lambda\). Both follow from the
endpoint integral being elementary, so that its singular part separates
exactly instead of asymptotically. Section~\ref{sec-appendix} does the
separation, and settles a cubic correction that is not as negligible as
it looks.

Figure~\ref{fig-slice} shows the equal-radius maximum along one slice.
It is a local maximum only: pushing the radii apart drives the
probability toward \(1\), and \((1, 5, 25)\) already gives 0.9353.

The value \(3/4\) does not characterize equal radii. The level set
\(P(\text{obtuse}) = 3/4\) is a curve in the two-dimensional space of
radius ratios that passes through the equal-radius point; for instance
the triple \((1, 2, r_3)\) reaches \(3/4\) at \(r_3 =\) 2.9188.

The obtuse probability is also not monotone in the radii.
Figure~\ref{fig-slice} shows the slice \((1, 1, \rho)\): it rises to
\(3/4\) at \(\rho = 1\), falls to a local minimum at
\(\rho = \sqrt{2}\), and rises again toward \(1\). The corner at
\(\sqrt{2}\) is exactly where the third term of Equation~\ref{eq-main}
switches off, by Corollary~\ref{cor-support}, and the kink is a
consequence of that term hitting the boundary of its support rather than
anything smooth going on in the other two. The corner has an exact size,
and it is a jump in the derivative rather than in the value: just below
the threshold the arcsine density near the top of its support gives
\(p_3 \approx (2 - \rho^2)/(2\pi)\), so the right derivative exceeds the
left by \(\sqrt{2}/\pi\). Numerically the two one-sided slopes are
\(-0.19758\) and \(0.25258\), a jump of \(0.45016\) against
\(\sqrt{2}/\pi = 0.45016\). Figure~\ref{fig-contour} shows the full
two-parameter picture, with the same ridge running through it.

\begin{figure}

\centering{

\pandocbounded{\includegraphics[keepaspectratio]{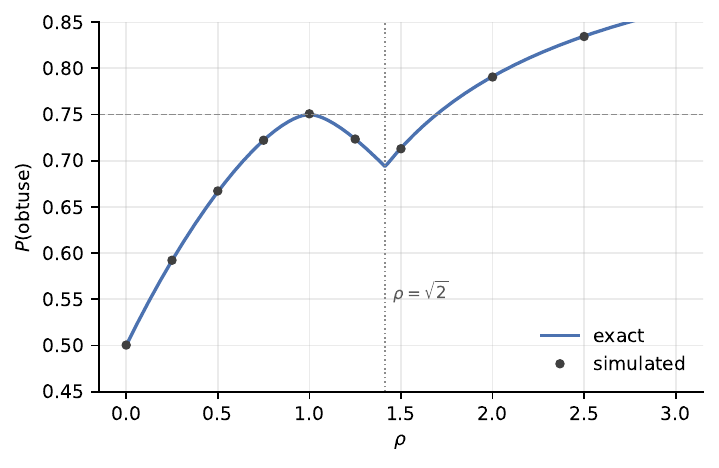}}

}

\caption{\label{fig-slice}Obtuse probability along the slice
\((1, 1, \rho)\). The curve is Equation~\ref{eq-main}; the points are
Monte Carlo estimates from \(10^6\) triangles each. The dashed lines
mark the equal-radius value \(3/4\) at \(\rho = 1\) and the threshold
\(\rho = \sqrt{2}\) beyond which the third vertex can no longer be the
obtuse one.}

\end{figure}%

\begin{figure}

\centering{

\pandocbounded{\includegraphics[keepaspectratio]{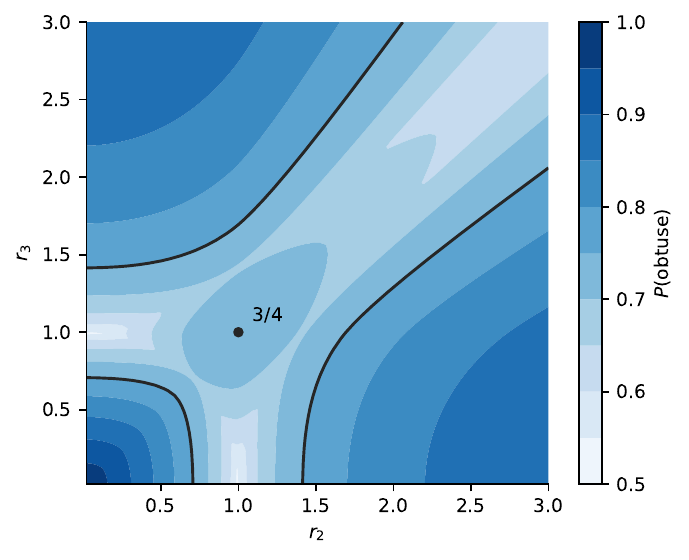}}

}

\caption{\label{fig-contour}Obtuse probability over the two-parameter
family with \(r_1 = 1\). The heavy contour is the level set \(3/4\),
which passes through the equal-radius point (marked) but is not confined
to it. The probability falls to \(1/2\) only in the corner where one
radius vanishes and the other two agree, and rises toward \(1\) as the
radii separate.}

\end{figure}%

Table~\ref{tbl-cases} collects the section. The first three rows are the
subfamilies where Equation~\ref{eq-main} closes, and the last two are
the range it takes over everything else; setting \(r_1 = r_2\) in the
second row returns \(1/2\), which is the minimum in the fourth.

\begin{longtable}[]{@{}
  >{\raggedright\arraybackslash}p{(\linewidth - 4\tabcolsep) * \real{0.3333}}
  >{\raggedright\arraybackslash}p{(\linewidth - 4\tabcolsep) * \real{0.3333}}
  >{\raggedright\arraybackslash}p{(\linewidth - 4\tabcolsep) * \real{0.3333}}@{}}
\caption{Closed forms and bounds for the special cases of
Equation~\ref{eq-main}.}\label{tbl-cases}\tabularnewline
\toprule\noalign{}
\begin{minipage}[b]{\linewidth}\raggedright
Configuration
\end{minipage} & \begin{minipage}[b]{\linewidth}\raggedright
\(P(\text{obtuse})\)
\end{minipage} & \begin{minipage}[b]{\linewidth}\raggedright
\end{minipage} \\
\midrule\noalign{}
\endfirsthead
\toprule\noalign{}
\begin{minipage}[b]{\linewidth}\raggedright
Configuration
\end{minipage} & \begin{minipage}[b]{\linewidth}\raggedright
\(P(\text{obtuse})\)
\end{minipage} & \begin{minipage}[b]{\linewidth}\raggedright
\end{minipage} \\
\midrule\noalign{}
\endhead
\bottomrule\noalign{}
\endlastfoot
\(r_1 = r_2 = r_3\) & \(3/4\) & Corollary~\ref{cor-equal} \\
\(r_3 = 0\), \(\ r_1 \leq r_2\) &
\(1 - \dfrac{1}{\pi}\arcsin\dfrac{r_1}{r_2}\) &
Corollary~\ref{cor-center} \\
\(r_1 = r_2 = 1\), \(\ r_3 = \rho \geq \sqrt{2}\) &
\(1 - \dfrac{4}{\pi^2}\chi_2\!\left(\dfrac{1}{\rho}\right)\) &
Corollary~\ref{cor-chi} \\
any admissible radii & \(1/2 \leq P < 1\) &
Proposition~\ref{prp-bounds} \\
\(r_3 = 0\), \(\ r_1 = r_2\) & \(1/2\), the minimum &
Proposition~\ref{prp-bounds} \\
\end{longtable}

\section{Every rotationally symmetric scheme is a
mixture}\label{sec-mixture}

Fixed concentric circles look like one sampling convention among the
many that Carroll's question admits, but they are the conditional kernel
for the whole rotationally symmetric class, which turns
Equation~\ref{eq-main} into a statement about every such convention at
once.

\begin{proposition}[Mixtures over the
radii]\protect\hypertarget{prp-mixture}{}\label{prp-mixture}

Let \(P_1, P_2, P_3\) be independent, each with a distribution invariant
under rotation about \(O\), and let \(R_k = \lVert P_k \rVert\) with
\(\Pr(R_k = 0) = 0\) for at least two of the three. Writing
\(P_{\text{obt}}(r_1, r_2, r_3)\) for the fixed-radius probability of
Equation~\ref{eq-main},
\begin{equation}\phantomsection\label{eq-mixture}{
P(\text{obtuse}) = \mathbb{E}\left[P_{\text{obt}}(R_1, R_2, R_3)\right] .
}\end{equation} Consequently the obtuse probabilities achievable across
the entire class are exactly the interval \([1/2, 1)\).

\end{proposition}

\emph{Proof.} A rotationally invariant law on the plane factors as
\(P_k = R_k(\cos\Theta_k, \sin\Theta_k)\) with \(\Theta_k\) uniform on
\([0, 2\pi)\) and independent of \(R_k\). Independence across \(k\)
makes \((\Theta_1, \Theta_2, \Theta_3)\) uniform on the torus and
independent of \((R_1, R_2, R_3)\), so conditionally on the radii the
configuration is exactly the one in Section~\ref{sec-setup}.
Conditioning and averaging gives Equation~\ref{eq-mixture}.

For the range, Proposition~\ref{prp-bounds} confines the integrand to
\([1/2, 1)\), so \(P(\text{obtuse}) \geq 1/2\) for every scheme in the
class. The upper end needs a little more than the convex hull, whose
closure would readmit \(1\): since \(1 - P_{\text{obt}}(R_1, R_2, R_3)\)
is strictly positive almost surely and bounded, its expectation is
strictly positive, so \(P(\text{obtuse}) < 1\). Both ends are approached
by deterministic radii: along the slice \((1, 1, \rho)\) the probability
is continuous in \(\rho\), equals \(1/2\) at \(\rho = 0\) by
Corollary~\ref{cor-center}, and tends to \(1\) as \(\rho \to \infty\) by
Corollary~\ref{cor-chi}, so the intermediate value theorem supplies
everything between. \(\square\)

The Gaussian convention is the case to test this against, and
Equation~\ref{eq-mixture} does better than reproduce its value: it
evaluates in closed form.

\begin{corollary}[The Gaussian
convention]\protect\hypertarget{cor-gauss}{}\label{cor-gauss}

Let \(R_1, R_2, R_3\) be independent Rayleigh variables, so that the
three vertices are independent standard bivariate normals. Then
\(\mathbb{E}\left[P_{\text{obt}}(R_1, R_2, R_3)\right] = 3/4\).

\end{corollary}

\emph{Proof.} For a standard bivariate normal \(R_k^2/2\) is
\(\mathrm{Exp}(1)\), and only ratios matter, so set \(E_k = R_k^2/2\)
and replace \(R_k^2\) by \(E_k\) throughout. By
Equation~\ref{eq-arcsine} and exchangeability of the \(E_k\), \[
P(\text{obtuse}) = \frac12 \sum_{k} \Pr(E_i U + E_j V > E_k)
  = \frac32 \Pr(E_1 U + E_2 V > E_3) .
\] Condition on \(U, V, E_1, E_2\). Since \(\Pr(E_3 > s) = e^{-s}\) and
\(\mathbb{E}[e^{-uE}] = 1/(1+u)\) for \(E \sim \mathrm{Exp}(1)\), \[
\Pr(E_1 U + E_2 V > E_3)
 = 1 - \mathbb{E}\left[e^{-E_1 U}\right]\mathbb{E}\left[e^{-E_2 V}\right]
 = 1 - \mathbb{E}\left[\tfrac{1}{1+U}\right]^2 .
\] For \(U\) arcsine on \([0,1]\) the standard integral \[
\int_0^1 \frac{du}{\pi\sqrt{u(1-u)}\,(u + a)} = \frac{1}{\sqrt{a(a+1)}}
\] gives \(\mathbb{E}[1/(1+U)] = 1/\sqrt{2}\) at \(a = 1\). The inner
probability is therefore \(1 - 1/2 = 1/2\), and
\(P(\text{obtuse}) = 3/4\). \(\square\)

This supplies a link the literature reports as missing. Edelman and
Strang (2015) note in their Section 1.1, and again in Section 3.6, that
the uniform-angle convention and the Gaussian convention both give
\(3/4\), and that they are ``not aware of an argument that links the
`angle picture' with the normal distribution''; they point to Portnoy
(1994) for related discussion of how often \(3/4\) arises. By the
identification in Section~\ref{sec-intro} the angle picture is the
equal-radius member of the family that Proposition~\ref{prp-mixture}
integrates, the Gaussian convention is the integral, and
Corollary~\ref{cor-gauss} evaluates it. The two agree because
\(\mathbb{E}[1/(1+U)] = 1/\sqrt{2}\).

Two cautions on how much that says. The equal-radius configuration
carries no mass under the Rayleigh law, so it is a member of the family
of kernels rather than a component of positive weight. And what the
computation matches is the obtuse probability, not the two shape
measures, which genuinely differ: the Gaussian law is far from uniform
on angle space, as Edelman and Strang (2015) emphasize.

Simulation confirms both sides: \(4 \times 10^6\) Gaussian triangles
give 0.7496 (standard error 0.0002), and averaging
Equation~\ref{eq-main} over 20,000 Rayleigh radius triples gives 0.7495
(standard error 0.0005). What makes Corollary~\ref{cor-gauss} surprising
is the spread of what is being averaged. Over those same draws the
fixed-radius probability ranged from 0.5141 to 0.9856, so a function
that swings across nearly half the unit interval averages back to
precisely the value it takes when the three radii coincide.

\section{Numerical evidence}\label{sec-numeric}

\subsection{Against the definition}\label{against-the-definition}

The simulation is a direct transcription of the definition: sample three
angles, form the three vertices, and test the sign of the three dot
products. No part of the derivation is reused, so agreement is a genuine
check.

\begin{Shaded}
\begin{Highlighting}[]
\NormalTok{theta }\OperatorTok{=}\NormalTok{ rng.uniform(}\DecValTok{0}\NormalTok{, }\DecValTok{2} \OperatorTok{*}\NormalTok{ np.pi, size}\OperatorTok{=}\NormalTok{(n, }\DecValTok{3}\NormalTok{))}
\NormalTok{pts }\OperatorTok{=}\NormalTok{ np.stack([radii }\OperatorTok{*}\NormalTok{ np.cos(theta), radii }\OperatorTok{*}\NormalTok{ np.sin(theta)], axis}\OperatorTok{={-}}\DecValTok{1}\NormalTok{)}
\NormalTok{np.mean(((pts[:, i] }\OperatorTok{{-}}\NormalTok{ pts[:, k]) }\OperatorTok{*}\NormalTok{ (pts[:, j] }\OperatorTok{{-}}\NormalTok{ pts[:, k])).}\BuiltInTok{sum}\NormalTok{(}\DecValTok{1}\NormalTok{) }\OperatorTok{\textless{}} \DecValTok{0}\NormalTok{)}
\end{Highlighting}
\end{Shaded}

Table~\ref{tbl-validation} compares Equation~\ref{eq-main} against
\(n = 2 \times 10^6\) simulated triangles for a range of radius triples,
including the degenerate case with a vertex at the center and cases
where one circle dominates. The last column is the standardized
difference, \((\hat{p} - p)\sqrt{n / (p(1-p))}\), which should behave
like a standard normal draw, and SE is the standard error of the
simulated estimate.

\begin{longtable}[]{@{}lrrrr@{}}

\caption{\label{tbl-validation}Exact obtuse probabilities from
Equation~\ref{eq-main} against Monte Carlo estimates from
\(n = 2 \times 10^6\) triangles per configuration.}

\tabularnewline

\toprule\noalign{}
\((r_1, r_2, r_3)\) & Exact & Simulated & SE & \(z\) \\
\midrule\noalign{}
\endhead
\bottomrule\noalign{}
\endlastfoot
1, 1, 1 & 0.750000 & 0.749823 & 0.000306 & -0.58 \\
1, 1, 0.5 & 0.666280 & 0.666729 & 0.000333 & +1.35 \\
1, 1, 1.2 & 0.731314 & 0.731217 & 0.000313 & -0.31 \\
1, 1, 2 & 0.791146 & 0.791095 & 0.000287 & -0.17 \\
1, 2, 3 & 0.757947 & 0.757672 & 0.000303 & -0.90 \\
2, 3, 4 & 0.711715 & 0.711651 & 0.000320 & -0.20 \\
0.5, 0.8, 1 & 0.690747 & 0.690962 & 0.000327 & +0.66 \\
1, 5, 25 & 0.935280 & 0.935172 & 0.000174 & -0.62 \\
1, 2, 0 & 0.833333 & 0.833414 & 0.000264 & +0.30 \\
0.3, 1, 1.05 & 0.604618 & 0.604937 & 0.000346 & +0.92 \\

\end{longtable}

All ten standardized differences sit within the range expected from ten
independent normal draws. The package's test suite compares the three
per-vertex probabilities separately, which is the stronger check: it
exercises not only Equation~\ref{eq-main} but also the decomposition in
Lemma~\ref{lem-additive} and the individual terms of
Equation~\ref{eq-pk}.

\subsection{The expansion at equal
radii}\label{the-expansion-at-equal-radii}

Table~\ref{tbl-local} checks Equation~\ref{eq-local} directly. The ratio
of the exact deficit to the predicted one approaches \(1\) slowly, like
\(1 + c/\log(1/\varepsilon)\), which is the signature of an
\(O(\lvert\alpha\rvert^2)\) remainder sitting under an
\(\lvert\alpha\rvert^2\log(1/\lvert\alpha\rvert)\) leading term.

\begin{longtable}[]{@{}lrrrr@{}}

\caption{\label{tbl-local}The expansion Equation~\ref{eq-local} against
exact values, for three perturbation directions
\(\alpha = \varepsilon d\) of the squared radii.}

\tabularnewline

\toprule\noalign{}
\(d\) & \(\varepsilon\) & Exact deficit & Equation~\ref{eq-local} &
Ratio \\
\midrule\noalign{}
\endhead
\bottomrule\noalign{}
\endlastfoot
\((0, 0, 1)\) & 0.01 & 3.039e-05 & 2.333e-05 & 1.3028 \\
\((0, 0, 1)\) & 0.001 & 4.202e-07 & 3.500e-07 & 1.2009 \\
\((0, 0, 1)\) & 0.0001 & 5.369e-09 & 4.666e-09 & 1.1507 \\
\((1, -1, 0)\) & 0.01 & 8.404e-05 & 6.297e-05 & 1.3347 \\
\((1, -1, 0)\) & 0.001 & 1.190e-06 & 9.796e-07 & 1.2151 \\
\((1, -1, 0)\) & 0.0001 & 1.540e-08 & 1.330e-08 & 1.1585 \\
\((3, -1, -2)\) & 0.01 & 4.818e-04 & 3.318e-04 & 1.4519 \\
\((3, -1, -2)\) & 0.001 & 7.247e-06 & 5.768e-06 & 1.2564 \\
\((3, -1, -2)\) & 0.0001 & 9.693e-08 & 8.218e-08 & 1.1795 \\

\end{longtable}

\subsection{The containment invariant}\label{the-containment-invariant}

Proposition~\ref{prp-center} makes a prediction that is easy to check
directly and easy to disbelieve: the triangle surrounds the origin with
probability \(1/4\) no matter how far apart the radii are pulled.

Simulation bears it out. Over \((1, 1, 1)\), \((1, 1, 0.5)\), and
\((1, 5, 25)\), whose obtuse probabilities differ by more than \(0.26\),
the origin fell inside the triangle in 0.2493, 0.2499, and 0.2504 of
\(10^6\) draws each.

\section{Software}\label{sec-software}

The \texttt{rtcc} package evaluates Equation~\ref{eq-quadrature} with
adaptive quadrature and the kinks supplied as break points, and includes
the simulator used above. Radii are unitless, and a radius of zero is
allowed.

\begin{Shaded}
\begin{Highlighting}[]
\ExtensionTok{$}\NormalTok{ uv run rtcc 1 1 1}
\ExtensionTok{radii}\NormalTok{            1, 1, 1}
\ExtensionTok{P}\ErrorTok{(}\ExtensionTok{obtuse}\KeywordTok{)}        \ExtensionTok{0.750000}
  \ExtensionTok{at}\NormalTok{ vertex 1    0.250000}
  \ExtensionTok{at}\NormalTok{ vertex 2    0.250000}
  \ExtensionTok{at}\NormalTok{ vertex 3    0.250000}

\ExtensionTok{$}\NormalTok{ uv run rtcc 1 2 3 }\AttributeTok{{-}{-}simulate}\NormalTok{ 1000000 }\AttributeTok{{-}{-}seed}\NormalTok{ 1}
\ExtensionTok{radii}\NormalTok{            1, 2, 3}
\ExtensionTok{P}\ErrorTok{(}\ExtensionTok{obtuse}\KeywordTok{)}        \ExtensionTok{0.757947}
  \ExtensionTok{at}\NormalTok{ vertex 1    0.472162}
  \ExtensionTok{at}\NormalTok{ vertex 2    0.285784}
  \ExtensionTok{at}\NormalTok{ vertex 3    0.000000}
\ExtensionTok{simulated}\NormalTok{        0.758088 }\ErrorTok{(}\ExtensionTok{SE}\NormalTok{ 0.000428}\KeywordTok{)}
\end{Highlighting}
\end{Shaded}

The \texttt{-\/-json} flag emits the same values as JavaScript Object
Notation (JSON) for scripting, and the Python interface exposes
\texttt{p\_obtuse}, \texttt{p\_vertices}, and \texttt{simulate}
directly. The accuracy quoted in Section~\ref{sec-main} is checked by
the tests \texttt{test\_vertex\_at\_the\_center} and
\texttt{test\_two\_equal\_radii\_closed\_form}, which compare the
quadrature against Corollary~\ref{cor-center} and
Corollary~\ref{cor-chi}; they assert agreement to \(10^{-9}\), well
inside the few units in the last place actually observed. Every number
and figure in this paper is computed at build time from the package
rather than transcribed. Source and tests are at
\url{https://github.com/bgreenwell/rtcc}.

\section{Discussion}\label{sec-discussion}

Giving each vertex its own circle keeps everything that made the
circular Pillow Problem well posed and adds two free parameters. The
answer stays elementary: Equation~\ref{eq-main} is a sum of three tail
probabilities for weighted sums of independent arcsine variables, and
the classical \(3/4\) is the point where all three thresholds coincide
and each tail is exactly one half.

The geometry that drops out of the algebra is the part I find most
useful. Proposition~\ref{prp-center} isolates what the generalization
actually breaks: the probability that the triangle surrounds the common
center never moves off \(1/4\), and it is the identification of that
center with the circumcenter, not the counting argument behind it, that
equal radii were quietly supplying. Corollary~\ref{cor-support} says the
obtuse angle can only sit at a vertex whose circle does not dominate the
other two in the Pythagorean sense, which is why the sum usually has two
nonzero terms rather than three. Proposition~\ref{prp-bounds} says no
concentric arrangement can make obtuse triangles rarer than even odds,
and identifies the single configuration that achieves it.

The agreement between the circular and Gaussian conventions, which
Edelman and Strang (2015) record as lacking a satisfying theoretical
link, is a consequence of the mixture identity and a single arcsine
integral. That is Corollary~\ref{cor-gauss}, and it is the result I
would lead with. Proposition~\ref{prp-mixture} is why the fixed-radius
formula is worth more than the family it was derived for. Carroll's
question has no answer until a convention is named, and the mixture
identity says that naming a rotationally symmetric convention is the
same as naming a distribution for the radii. The floor of \(1/2\) then
applies to all of them at once, and the agreement between the circle and
Gaussian conventions becomes a fact about an average rather than a
coincidence between two special cases.

Replacing the circles by spheres in \(\mathbb{R}^d\) changes the
conditional half-space probability in Equation~\ref{eq-conditional} from
an arcsine to an incomplete beta. Two parts of that extension should be
distinguished. The mixture identity Proposition~\ref{prp-mixture} is a
routine transcription: its proof uses only that a rotationally invariant
law factors into a radius and an independent uniform direction, which
holds in every dimension, so the conditional-kernel reading of the
fixed-radius family carries over unchanged. What is genuinely open is
the fixed-radius formula itself, since Lemma~\ref{lem-omega} and
Lemma~\ref{lem-signed} are planar arguments and I do not know what
replaces the law of sines on a sphere. The target is known in the
Gaussian case, where Edelman and Strang (2015) give
\(3(1 - I(3/4, d/2, d/2))\) and random triangles become acute with
probability tending to one as \(d\) grows (Eisenberg and Sullivan 1996),
so a \(d\)-dimensional analogue of Equation~\ref{eq-main} would have
something definite to reproduce. Dropping independence between the three
vertices is the other direction, and the harder one, since
Lemma~\ref{lem-additive} is all that survives.

\appendix

\section{The endpoint expansion}\label{sec-appendix}

This appendix supplies the two steps Proposition~\ref{prp-local} quoted,
and disposes of a cubic correction that enters at the same order as the
term Equation~\ref{eq-local} keeps.

\begin{lemma}[Exact endpoint
decomposition]\protect\hypertarget{lem-endpoint}{}\label{lem-endpoint}

Fix \(\delta > 0\) and let \[
J(\mu, \eta) = \int_0^{\delta}
  \left(x - \sqrt{\left[\mu x^2 - \eta\right]^{+}}\right) dx .
\] Then for \(\lvert\eta\rvert < \mu\delta^2\),
\begin{equation}\phantomsection\label{eq-endpoint}{
J(\mu, \eta) = \frac{\eta}{4\sqrt{\mu}}\log\frac{1}{\lvert\eta\rvert}
  + A(\mu, \eta) ,
}\end{equation} where \(A\) is real analytic on a neighborhood of
\((1, 0)\).

\end{lemma}

\emph{Proof.} The integrand is elementary. Writing
\(F(x) = \tfrac{x}{2}\sqrt{\mu x^2 - \eta}
 - \tfrac{\eta}{2\sqrt\mu}\log\!\left(\sqrt\mu\, x + \sqrt{\mu x^2 - \eta}\right)\)
for an antiderivative of \(\sqrt{\mu x^2 - \eta}\), integration for
\(\eta > 0\) splits at \(x_0 = \sqrt{\eta/\mu}\), where the positive
part switches on, and gives \(J = \delta^2/2 - F(\delta) + F(x_0)\).
Since \(\sqrt{\mu x_0^2 - \eta} = 0\) and
\(\sqrt\mu\, x_0 = \sqrt{\eta}\), \[
F(x_0) = -\frac{\eta}{2\sqrt\mu}\log\sqrt{\eta}
       = \frac{\eta}{4\sqrt\mu}\log\frac{1}{\eta} ,
\] which is the displayed singular term; the case \(\eta < 0\) has no
switch and the same computation applies at \(x = 0\). What remains,
\(A = \delta^2/2 - F(\delta)\), involves \(\sqrt{\mu\delta^2 - \eta}\)
and \(\log(\sqrt\mu\,\delta + \sqrt{\mu\delta^2 - \eta})\), both real
analytic in \((\mu, \eta)\) wherever \(\mu\delta^2 > \lvert\eta\rvert\)
and \(\mu > 0\). \(\square\)

Equation~\ref{eq-endpoint} is what those two steps needed. Because \(A\)
is analytic, the endpoint contributions are differentiable in \(\alpha\)
once the explicit \(\eta\log(1/\lvert\eta\rvert)\) terms are removed,
which is exactly the differentiability of \(\Lambda\); and because
\(A\)'s derivatives are bounded on a fixed neighborhood, the resulting
\(O(\lvert\alpha\rvert^2)\) bound holds with a constant independent of
the direction of \(\alpha\). Two approximations were used to reach
\(J\), namely \(\arccos t - \arccos t_0 = t_0 - t + O(t^3)\) and
\(\sin x = x + O(x^3)\); both are Taylor expansions with remainders
uniform on \([0, \delta]\). Since the \(\alpha_k\) can be centered,
\(\max_{k<l}\lvert e_{kl}\rvert\) is comparable to
\(\lvert\alpha\rvert\), so the displayed sum in Equation~\ref{eq-local}
is of exact order \(\lvert\alpha\rvert^2\log(1/\lvert\alpha\rvert)\).

The cubic Taylor corrections need one further step, because at first
sight they are not small enough to ignore. They lead to integrals of
\(\left[\mu x^2 - \eta\right]_+^{3/2}\), whose singular part is
\(\frac{3\eta^2}{16\sqrt\mu}\log\frac{1}{\lvert\eta\rvert}\) by the same
antiderivative used in Lemma~\ref{lem-endpoint}, contributing
\(-\frac{1}{32\sqrt\mu}\eta^2\log\frac{1}{\lvert\eta\rvert}\) to the
endpoint. That is the order of the terms Equation~\ref{eq-local}
retains.

A second correction of the same order cancels it, and the two must be
taken together. Reaching \(J\) also replaced \(\sin x\) by \(x\), and
\(\sin^2 x = x^2 - x^4/3 + O(x^6)\), so the square root carries an extra
\(-\mu x^4 / \left(6\sqrt{\mu x^2 - \eta}\right)\). Now
\(\int x^4 (x^2 - \eta)^{-1/2} dx\) has log coefficient \(3\eta^2/8\),
giving \(+\frac{1}{32\sqrt\mu}\eta^2\log\frac{1}{\lvert\eta\rvert}\),
which is the negative of the first. The two same-order corrections
annihilate each other at each endpoint separately, so no
\(\eta^2\log(1/\lvert\eta\rvert)\) term ever reaches the sum and the
coefficient in Equation~\ref{eq-local} is the one
Lemma~\ref{lem-endpoint} already gives. The higher terms of both series
enter at order \(\lvert\alpha\rvert^3\log(1/\lvert\alpha\rvert)\).

Two checks. Numerically, the even-in-\(\eta\) part of a single endpoint
is \(-\eta^2/2\) with no logarithm: the ratio of that part to \(\eta^2\)
holds at \(-0.500\) across \(\eta = 10^{-3}, 10^{-4}, 10^{-5}\), whereas
an uncancelled \(\eta^2\log\) term would make it grow like
\(\log(1/\eta)\). Structurally, at \(\mu = 1\) the reflection
\(x \mapsto \pi/2 - x\) turns \(\sin^2 x\) into \(1 - \sin^2 x\), whence
\(\arccos\sqrt{\cos^2 x - \eta} = \arcsin\sqrt{\sin^2 x + \eta}\), and
since \(\arcsin u + \arccos u = \pi/2\) pointwise, including where the
arguments are clipped, \[
I(1, \eta) + I(1, -\eta) = \frac{\pi^2}{4} ,
\] so \(I(1, \cdot) - \pi^2/8\) is odd and every even-order term in
\(\eta\), the \(\eta^2\log\) term included, vanishes identically there.

\section*{References}\label{references}
\addcontentsline{toc}{section}{References}

\phantomsection\label{refs}

\end{document}